\documentclass[11pt]{amsart}
\usepackage{xcolor}
\usepackage[a4paper,asymmetric]{geometry}
\usepackage{latexsym}
\usepackage{epsfig}
\usepackage{amsthm,euscript}
\usepackage{amssymb}
\usepackage{amsfonts}
\usepackage{amsmath}
\usepackage{epstopdf}
\newtheorem{theorem}{Theorem}[section]
\newtheorem{thm}[theorem]{Theorem}
\newtheorem{lemma}[theorem]{Lemma}
\newtheorem{lem}[theorem]{Lemma}
\newtheorem{proposition}[theorem]{Proposition}

\newtheorem{corollary}[theorem]{Corollary}

\theoremstyle{definition}
\newtheorem{definition}[theorem]{Definition}
\newtheorem{defn}[theorem]{Definition}

\theoremstyle{remark}
\newtheorem{remark}[theorem]{Remark}
\newtheorem{rem}[theorem]{Remark}
\numberwithin{equation}{section}

 \DeclareMathAlphabet{\mathpzc}{OT1}{pzc}{m}{it}

\def\al{\alpha}

 \newcommand{\EE}{\mathbb{E}}

 \newcommand{\RR}{\mathbb{R}}
 
 \newcommand{\PP}{\mathbb{P}}
 \newcommand{\HH}{\mathbb{H}}

 \newcommand{\Be}{\begin{equation}}
 \newcommand{\Ee}{\end{equation}}
 \newcommand{\Bs}{\begin{split}}
 \newcommand{\Es}{\end{split}}
  \newcommand{\Bes}{\begin{equation*}}
 \newcommand{\Ees}{\end{equation*}}
 \newcommand{\BT}{\begin{thm}}
 \newcommand{\ET}{\end{thm}}
 \newcommand{\Bp}{\begin{proof}}
 \newcommand{\Ep}{\end{proof}}
 \newcommand{\BL}{\begin{lem}}
 \newcommand{\EL}{\end{lem}}
 \newcommand{\BP}{\begin{proposition}}
 \newcommand{\EP}{\end{proposition}}
 \newcommand{\BC}{\begin{corollary}}
 \newcommand{\EC}{\end{corollary}}
 \newcommand{\BR}{\begin{rem}}
 \newcommand{\ER}{\end{rem}}
 \newcommand{\BD}{\begin{defn}}
 \newcommand{\ED}{\end{defn}}
 \newcommand{\BI}{\begin{itemize}}
 \newcommand{\EI}{\end{itemize}}

\begin{document}
\title
[Smooth densities for SDEs driven by SBM with Markovian switching]
{Smooth densities for SDEs driven by 
%subordinated Brownian motion with Markovian switching}
subordinate Brownian motion with Markovian switching}

\author[X. Sun ]{Xiaobin {\sc Sun}}
\address{School of Mathematics and Statistics\\Jiangsu Normal University\\
Xuzhou 221116, China}
\email{xbsun@jsnu.edu.cn}

\author[Y. Xie ]{Yingchao {\sc Xie}}
\address{School of Mathematics and Statistics\\Jiangsu Normal University\\
Xuzhou 221116, China}
\email{ycxie@jsnu.edu.cn}

%\maketitle

\begin{abstract} \label{abstract}
%This paper considers a class of stochastic differential equations driven by subordinated Brownian motion with Markovian switching. Using Malliavin calculus, we study the smoothness of the density for the solution under uniform H\"ormander's type condition.
In this paper we consider a class of stochastic differential equations driven by subordinate
Brownian motion with Markovian switching.  We use Malliavin calculus to study the smoothness of the density for the solution under uniform H\"ormander's type condition.
\end{abstract}

\maketitle

\section{Introduction}

%This paper considers  the  following jump-diffusion with Markovian switching on $\RR^n$:
In this paper we consider the following jump-diffusion with Markovian switching in $\RR^n$:
\begin{equation}
dX_t=b(X_t, \alpha_t)dt+\sigma dL_t,\ (X_0\,, \al_0)=(x,\al)\in \RR^n \times \mathbb{S} \,, \label{e.x}
\end{equation}
where $b: \RR^n\times \mathbb{S}\rightarrow \RR^n$ is 
%an appropriate function, 
a function satisfying certain functions to be specified below, 
$\sigma$ is a $n\times d$ constant matrix, $L_t$ is a
$d$-dimensional subordinated Brownian motion, $\mathbb{S}=\{1, 2, \ldots, m\}$ and  $\left\{\alpha_t\,, t\ge 0\right\}$ is a right-continuous
$\mathbb{S}$-valued Markov chain  described by
\begin{equation}
\mathbb{P}\{\alpha_{t+\Delta}=j|\alpha_{t}=i\}=\left\{\begin{array}{l}
\displaystyle q_{ij}\Delta+o(\Delta),\quad i\neq j\\
1+q_{ii}\Delta+o(\Delta),\quad i=j,\end{array}\right.
\label{alpha}
\end{equation}
and $Q= (q_{ij})_{1 \le i,j \le m}$ is a $Q$-matrix.

In recent years, there has been increasing interest in stochastic differential equations (SDEs) 
%with Markovian switching, such as existence and uniqueness of the solution, the existence of the invariant measure, the stability and other important properties have been studied (See, e.g.\cite{BBG, M, X, YZ, YM}).
with Markovian switching. Among the properties studied are the existence and uniqueness of the solutions, the existence of the invariant measure and stability (see \cite{BBG, M, X, YZ, YM}).
However, the smoothness of the densities of the solutions to this kind of SDEs have not been studied much.
When the noise is Brownian motion, the smoothness of the densities of the solutions has been proved
under the uniform H\"{o}rmander's condition in \cite{HNSX}.

The main purpose of this paper is to study the smoothness of the density of the solution of equation (\ref{e.x}). 
%For the reason of technique, we only consider the additive noise here.
For technical reasons, we only consider the additive noise here.
In order to show the smoothness of density for $X_t$, we need to develop Malliavin calculus for $X_t$ and show the Malliavin covariance matrix has all negative moments.
The difficulty here is the appearance of the switching term $\al_t$.
Our procedure is to follow the method in \cite{HNSX}, i.e., we will perform perturbations of the underlying Brownian motion, keeping the Markovian switching process $\alpha_t$ and the subordinator unperturbed. 
%The technique for this analysis is inspired in the partial Malliavin calculus, which can be regarded
%as a stochastic calculus of variation for random variables with values in  a Hilbert space.
The technique for this analysis can be regarded as a stochastic calculus of variation for random variables with values in  a Hilbert space and is partly inspired by Malliavin calculus.

When the switching term $\alpha_t$ is not present, Kusuoka \cite{KS} proved the solution
has a smooth density under a nondegenerate condition on $\sigma$. Zhang \cite{Z1} established that the solution has a smooth density in a special degenerate case n \cite{Z1}  and under the uniform H\"{o}rmander's type condition in \cite{Z2}. When the switching term $\alpha_t$ is present, 
%it will make some troubles by the fact of $X_t$ depends on the jump process $\alpha_t$, we shall use the strategy inspired by \cite{FLT, HNSX}.
things are more complicated due the fact that $X_t$ depends on the jump process $\alpha_t$.
We will use a strategy inspired by \cite{FLT, HNSX}.
More precisely, we first notice that the jump times of $\al_t$ form a subset of the jump times of some Poisson process $N_t$,
independent of the driving subordinate Brownian motion $L_t$. Then conditioning on $N_t=k$, there exists a random interval
$[T_1, T_2)$ with $0\leq T_1< T_2\leq t$, such that $T_2-T_1\geq\frac{t}{k+1}$ and $\alpha_t=\alpha_{T_1}, t\in[T_1, T_2)$.
On this time interval, we will follow the procedure developed in \cite{Z2}. 
%This requires a version of the Norris' type lemma developed in \cite{Z1, Z2} on time intervals, 
This requires a version of Norris' lemma developed in \cite{Z1, Z2} on time intervals, 
%which is a key step to study the Malliavin covariance matrix has all negative moments. 
which is a key tool to show that the Malliavin covariance matrix has all negative moments,
%As a consequence, the solution $X_t$ has smooth density.
which implies that the solution $X_t$ has a smooth density.

The paper is organized as follows. In the next section, we introduce  some notation
and  assumptions that we use throughout the paper. The Malliavin calculus for $X_t$ is developed in Section 3.
In  Section 4, we first develop a Norris' type lemma on time interval, then use it to show that, 
under a uniform H\"{o}rmander type condition, the Malliavin covariance matrix has all negative finite moments. Finally, we prove that $X_t$ has a smooth density by considering the small jumps and the large jumps separately.

%Along the paper 
In this paper,
$C$ will denote a generic constant which may vary from line to line and it might depend on $T$, the exponent $p\ge 2$, the initial condition $x$ and a fixed element $h\in H$ (the precise definition of $H$ is in Section 3).

\section{Preliminaries}

\begin{itemize}
\item{ Let $(\Omega_1, \mathcal{F}_1, \mathbb{P}_1)$ be the $d$-dimensional
canonical Wiener space.
That is,  $\Omega_1$  is  the set of all continuous maps $\omega_1:\RR_{+}\rightarrow \mathbb{R}^d$ such that
$\omega_1(0)=0$ and $\mathbb{P}_1$ is the canonical Wiener measure such that coordinate process
$$
W_t(\omega_1):=\omega_1(t)
$$
is a standard $d$-dimensional Brownian motion.}

\item{Let$(\Omega_2, \mathcal{F}_2, \mathbb{P}_2)$ be the space of all increasing, purely discontinuous and c\`{a}dl\`{a}g
functions from $\RR_{+}$ to $\RR_{+}$ with $\omega_2(0)=0$, which is endowed with the Skorohod metric and the
probability measure $\PP_2$ so that the coordinate process
$$
S_t(\omega_2):=\omega_2(t)
$$
is an increasing one dimensional L\'{e}vy process (called a subordinator) on $\RR_{+}$ with Laplace transform:
$$
\EE_2 e^{-s S_t}=\exp\left\{t\int^{\infty}_{0}(e^{-su}-1)\nu_S (du)\right\},
$$
where $\EE_2$ is the expectation with respect to $\PP_2$, $\nu_S$ is the L\'{e}vy measure satisfying $\nu_S(\{0\})=0$ and
$$
\nu_{S}((-\infty, 0])=0,\quad \int^{\infty}_{0}(1\wedge u)\nu_{S}(du)<\infty.
$$
}

\item{Let $(\Omega_3, \mathcal{F}_3, \mathbb{P}_3)$ be a complete probability space, on which $\left\{\alpha_t\,, t\ge 0\right\}$ is a right-continuous
$\mathbb{S}$-valued Markov chain satisfying (\ref{alpha}).
}
\end{itemize}

We will use $(\Omega,\mathcal{F}, \mathbb{P})$ to denote the product probability space
$(\Omega_1\times\Omega_2\times\Omega_3, \mathcal{F}_1\times\mathcal{F}_2\times\mathcal{F}_3,\mathbb{P}_1\times\mathbb{P}_2\times\PP_3)$.
We extend  $W_t$, $S_t$ and $\alpha_t$ to random variables on $\Omega$ by letting
$W_t(\omega)=\omega_{1}(t)$, $S_t(\omega)=\omega_2(t)$ and $\alpha_t(\omega)=\alpha_t(\omega_3)$, respectively, if $\omega=(\omega_1, \omega_2, \omega_3)$. Thus on $(\Omega,\mathcal{F}, \mathbb{P})$, $W_t$, $S_t$ and $\alpha_t$ are independent. We define
$$
L_t(\omega):=W_{S_t}(\omega)=\omega_1(\omega_2(t)).
$$
Then $(L_t)_{t\geq 0}$ is a L\'{e}vy process (called a subordinate Brownian motion) with characteristic function:
$$
\EE e^{i \langle z,L_t\rangle_{\RR^d}}=\exp\left\{t\int_{\RR^d}(e^{i\langle z,y\rangle_{\RR^d}}-1-i\langle z,y\rangle_{\RR^d}1_{|y|\leq 1})\nu_{L}(dy)\right\},
$$
where $\EE$ is the expectation with respect to $\PP$, $\nu_{L}$ is the L\'{e}vy measure given by
$$
\nu_{L}(\Gamma)=\int^{\infty}_{0}(2\pi s)^{-d/2}\left(\int_{\Gamma}e^{-\frac{|y|^2}{2s}}dy\right)\nu_{S}(ds),\quad \Gamma\in \mathcal{B}(\RR^d).
$$
Obviously, $\nu_{L}$ is a symmetric measure.

\vspace{3mm}

Let $\mathbb{S}=\{1, 2, \ldots,m\}$,  where $m$ is a given positive integer which will be fixed
throughout the paper. The matrix $Q=(q_{ij})$ is assumed to satisfy the following assumptions.
\begin{itemize}
\item[(i)] $q_{ij}\geq0$ for $i\neq j$,
\item[(ii)] $q_{ii}=-\sum_{j\neq i}q_{ij}$ for $i\in \mathbb{S}$,
\item[(iii)] $\sup_{i, j\in\mathbb{S}}|q_{ij}|:=K<\infty$.
\end{itemize}

It is well known (see \cite{BBG}) that the process $\left\{\al_t\,, 0\le t\le T\right\}$
can be described as follows.
Let $g: \mathbb{S}\times [0,
m(m-1)K]\rightarrow\mathbb{R}$ be defined by
\[
g(i, z)=\sum_{j\in \mathbb{S}\backslash i }(j-i)1_{z\in\triangle_{ij}},\quad \forall \ i\in \mathbb{S}\,,
\]
where $\triangle_{ij}$'s are the consecutive (with respect to the
lexicographic ordering on $\mathbb{S}\times\mathbb{S}$) left-closed,
right-open intervals of $\mathbb{R}_{+}$, each having length
$q_{ij}$, with $\Delta_{12}= [0,q_{12})$. Then,
(\ref{alpha}) can be rewritten as
\begin{equation} \label{2.3}
\alpha_t=\al+\int^t_0\int_{[0, m(m-1)K]}g(\alpha_{s-}, z)N(ds,dz),
\end{equation}
where $N(dt, dz)$ is a Poisson random measure defined on $\Omega\times\mathcal{B}(\mathbb{\RR_{+}})\times\mathcal{B}(\mathbb{\RR_{+}})$, whose intensity measure is the Lebesgue measure, and $N(dt, dz)$ is independent of $W_t$ and $S_t$.

\vspace{0.3cm}

For $k\in\mathbb{N}$ we denote by $C^{k}(\mathbb{R}^{n}\times\mathbb{S};\mathbb{R}^n)$
the family of all $\mathbb{R}^n$-valued functions $f(x, \alpha)$ on $\mathbb{R}^{n}\times\mathbb{S}$ which
are $k$-times continuously differentiable in $x$ for any $\alpha \in \mathbb{S}$.  The  $k$-th derivative tensor of
$f$ with respect to $x$ is denoted by $\nabla^k f(x, \alpha)$.

For $x\in \mathbb{R}^n$ and $\sigma\in\mathbb{R}^n\times\mathbb{R}^d$, we use the notation $|x|^2=\sum^{n}_{i=1}|x_i|^2$
and $|\sigma|^2=\sum^{n}_{i=1}\sum^{d}_{j=1}|\sigma_{ij}|^2$. We will consider the metric  $\Lambda$ on $\mathbb{R}^n\times\mathbb{S}$  given by $\Lambda((x, i), (y, j))=|x-y|+d(i, j)$,  for $x, y\in\mathbb{R}^n, i, j\in\mathbb{S}$, where $d(i, j)=0$ if $i=j$ and $d(i, j)=1$ if $i\neq j$.

Now we make the following assumptions on the function  $b:\mathbb{R}^{n}\times\mathbb{S}\to\mathbb{R}^{n}$.

\smallskip
\noindent
\textbf{(H1)} There is a positive constant $C_1$ such that
\[
|b(x, i)-b(y, i)|\leq C_1|x-y| \quad\hbox{for any $x,y\in\mathbb{R}^{n}, i\in\mathbb{S}$}\,.
\]

\smallskip
\noindent
\textbf{(H2)}
The function $b$ belongs to
$C^{2}(\mathbb{R}^{n}\times \mathbb{S};\mathbb{R}^n)$ and its first order and second order
partial derivatives are bounded.

\medskip
It is clear that \textbf{(H2)} implies \textbf{(H1)}. 
%Similar as the argument in \cite{X}, 
Using an argument similar to that of \cite{X}, 
under the condition \textbf{(H1)}, for any initial value $(x,\alpha)\in \RR^n\times \mathbb{S}$, it is easy to prove equation (\ref{e.x}) has a unique strong solution $\{X_t, t\geq 0\}$,\quad i.e.,
$$
X_t=x+\int^t_0 b(X_s, \alpha_s)ds+\sigma L_t,\quad \mathbb{P}-a.s..
$$
If we consider the natural filtration:
$$
\mathcal{F}_t:=\sigma\{L_r, S_r, \alpha_r: 0\leq r\leq t\},
$$
then the solution $(X_t, \al_t)$ is a Markov process and the associated Markov semigroup $P_t$ satisfies
$$P_tf(x, \alpha)=\mathbb{E}f(X_t(x), \alpha_t(\al)),~~~~~ t>0, f\in \mathcal{B}_b(\mathbb{R}^n\times\mathbb{S}),$$
where $\mathcal{B}_b(\mathbb{R}^n\times\mathbb{S})$ be the family of all bounded Borel measurable functions on $\mathbb{R}^n\times\mathbb{S}$.

\section{The Malliavin calculus}

In this section we analyze the regularity, in the sense of Malliavin calculus, of 
%the random vector $X_t$ solution to the system
the solution $X_t$ to the system
(\ref{e.x}) and (\ref{2.3}). Denote by $H$ the Hilbert space $H=L^{2}([0, \infty); \mathbb{R}^{d})$, equipped with the inner
product $\langle h_1, h_2\rangle_H=\int^{\infty}_{0} \langle h_1(s),h_2(s)\rangle_{\RR^d} ds$.

For a  Hilbert space $U$ and a real number $p\ge 1$, we denote by $L^p(\Omega_1; U)$ the space of $U$-valued
random variables $\xi$ such that $\mathbb{E}_1\|\xi\|^{p}_{U}<\infty$, where $\mathbb{E}_1$
is the expectation in the probability space $(\Omega_1,\mathcal{F}_1, \mathbb{P}_1)$. We also set
$L^{\infty-}(\Omega_1; U):=\cap_{p<\infty}L^p(\Omega_1; U)$.

We introduce the derivative operator for a random variable $F$ in the space  $L^{\infty-}(\Omega_1; U)$ following
the approach of Malliavin in  \cite{Ma}. We say that $F$ belongs to $\mathbb{D}^{1,\infty}(U)$ if
there exists $DF\in L^{\infty-}(\Omega_{1}; H\otimes U)$ such that for any $h\in H$,
$$
\lim_{\varepsilon \rightarrow 0}\mathbb{E}_1\left\|\frac{F(\omega_1+\varepsilon \int^{\cdot}_0 h_s ds)
-F(\omega_1)}{\varepsilon }-\langle DF, h\rangle_H\right\|^{p}_{U}=0
$$
holds for every $p\geq 1$. In this case, we define the Malliavin derivative of $F$ in the direction $h$ by $D^{h}F :=\langle DF, h\rangle_{H}.$
Then, for any $p\ge 1$ we define  the Sobolev space $\mathbb{D}^{1,p}(U)$ as the completion of $\mathbb{D}^{1,\infty}(U)$ under the  norm
$$
\|F\|_{1, p, U}=\left[\mathbb{E}_1(\|F\|^p_{U})\right]^{1/p}
+\left[\mathbb{E}_1(\|DF\|^p_{H\otimes U})\right]^{1/p}.
$$
By induction we define the $k$th derivative by $D^{k}F=D(D^{k-1}F),$ which is a random element with values in  $H^{\otimes k}\otimes U$.
For any integer $k \ge 1$, the Sobolev space  $\mathbb{D}^{k,p}(U)$ is the  completion of  $\mathbb{D}^{k,\infty}(U)$ under the norm
$$
\|F\|_{k, p, U}=\|F\|_{k-1, p, U}+\|D^{k}F\|_{1, p, H^{\otimes k}\otimes U}.
$$
We denote $\mathbb{D}^{\infty}(U)=\cap_{k\geq1}\mathbb{D}^{k,\infty}(U)$.
It turns out that  $D$ is a closed operator from $L^{p}(\Omega_{1};U)$ to $L^{p}(\Omega_{1}; H\otimes U)$.
Its adjoint $\delta$ is called the divergence operator, and is continuous form $L^{p}(\Omega_{1}; H\otimes U)$
to $L^{p}(\Omega_{1};U)$  for any $p>1$. The duality relationship reads
\[
\mathbb{E}_1(\langle DF, u\rangle_{H\otimes U})=\mathbb{E}_1(\langle F,\delta(u)\rangle_{U}),
\]
for any $F\in  \mathbb{D}^{1,2}(U)$ and $u\in {\mathcal D}(\delta)$ which is the domain of $\delta$.

A square integrable random variable $F\in L^2(\Omega)$ can be identified with an element
of $L^2(\Omega_1; V)$, where $V=L^2(\Omega_2\times\Omega_3)$.

\vspace{0.3cm}

%For the reason of technique, 
For technical reasons, 
we always assume $S_t$ has finite moments (i.e. $\EE |S_t|^p<\infty$, for all $p\geq 1, t>0$) in this section. 
%Then we will use the similar argument as in \cite[Section 3.3]{Z1} to deal with the general case in subsection 4.2.2. 
We will argument similar to that of \cite[Section 3.3]{Z1} to deal with the general case in Subsection 4.2.2.

\subsection{Malliavin differentiability of solution}

Let $X^{\varepsilon  h}_t$ be the solution of equation
(\ref{e.x}) with $W_{S_t}$ replaced by $W_{S_t}+\varepsilon \int^{S_t}_0 h_s ds$, where $\varepsilon\in (0,1)$, that is,
\begin{eqnarray}\left\{\begin{array}{l}
\displaystyle dX^{\varepsilon  h}_t=b(X^{\varepsilon  h}_t, \alpha_t)dt+
\sigma dW_{S_t}+\varepsilon \sigma d\left(\int^{S_t}_{0}h_s ds\right),\\
(X^{\varepsilon  h}_0, \alpha_0)=(x, \alpha)\in\mathbb{R}^n\times\mathbb{S},
\end{array}\right.\label{4.2}
\end{eqnarray}
where $\alpha_t$ is defined by (\ref{alpha}). Thus, we have
\begin{eqnarray*}
\frac{X^{\varepsilon  h}_{t}-X_t}{\varepsilon }
=\!\!\!\!\!\!\!\!&&\frac{1}{\varepsilon }\int^t_0[b(X^{\varepsilon h}_s, \alpha_s)-b(X_s, \alpha_s)]ds
+\sigma\int^{S_t}_{0}h_s ds.
\end{eqnarray*}

In order to prove the Malliavin differentiability of the solution, we first give some preliminary lemmas.

\begin{lemma} \label{l.3.2} Suppose that condition \textbf{(H1)} holds. Then for any $T>0$, $h\in H$ and $p\geq 2$, we have
$$
\mathbb{E}\left[\sup_{t\leq T}|X^{\varepsilon  h}_t|^{p}\right]\leq C.
$$
\end{lemma}

\begin{proof} From (\ref{4.2}) it is easy to see that
\begin{eqnarray*}
|X^{\varepsilon  h}_t|^{p}
\leq\!\!\!\!\!\!\!\!&&  C\left[|x|^{p}+\left|\int^t_0 b(X^{\varepsilon  h}_s, \alpha_s)ds\right|^{p}
+|\sigma|^p |W_{S_t}|^{p}+ \varepsilon ^p\left|\int^{S_t}_0\sigma h_sds\right|^{p}\right]\\
:=\!\!\!\!\!\!\!\!&&C \left[ |x|^p+I_1(t)+I_2(t)+I_3(t)\right] \,.
\end{eqnarray*}
By the condition \text{(H1)}, H\"{o}lder's inequality and the fact that $S_t$ has finite moments
of all orders, we obtain
\[
\mathbb{E}\left[\sup_{t\leq T}\left(I_1(t) + I_2(t) + I_3(t)\right)  \right]
\leq C \int^T_0(\mathbb{E}|X^{\varepsilon  h}_s|^{p}+1)ds \,.
\]
Then the desired estimate follows from  Gronwall's lemma.
\end{proof}

\begin{lemma} \label{l.3.5}Suppose that condition \textbf{(H1)} holds.
Then for any $T>0$, $h\in H$ and $p\geq 2$, we have
\[
\mathbb{E}\left[\sup_{0\leq t\leq T}|X^{\varepsilon  h}_t-X_t|^{p}\right]\leq C\varepsilon ^{p}.
\]
\end{lemma}

\begin{proof} We write
\begin{eqnarray*}
X^{\varepsilon  h}_t-X_t
=\!\!\!\!\!\!\!\!&&\int^t_0 [b(X^{\varepsilon  h}_s, \alpha_s)-b(X_s, \alpha_s)]ds
+\varepsilon \int^{S_t}_0\sigma h_sds.
\end{eqnarray*}
Applying H\"{o}lder's  inequalities and the fact that $S_t$ has finite moments of all orders, we obtain
\begin{eqnarray*}
\mathbb{E}\left[\sup_{0\leq t\leq T}|X^{\varepsilon  h}_t-X_t|^{p}\right]
\leq\!\!\!\!\!\!\!\!\!\!&&C \int^T_0\mathbb{E}|X^{\varepsilon  h}_t-X_t|^{p}dt+C \varepsilon ^{p} .
\end{eqnarray*}
Hence Gronwall's inequality implies
$$
\mathbb{E}\left[\sup_{0\leq t\leq T}|X^{\varepsilon  h}_t-X_t|^{p}\right]\leq C \varepsilon ^p,
$$
which completes the proof.
\end{proof}

The following theorem is the main result of this subsection.
\begin{theorem}\label{main}   Suppose that  condition \textbf{(H2)} holds. For any $t>0$, $h\in H$, we have
$X_t\in \mathbb{D}^{1,\infty}(\mathbb{R}^n\otimes V)$ and $D^hX_t$ satisfies
\begin{equation}\left\{\begin{array}{l}
\displaystyle dD^hX_t=\nabla b(X_t, \alpha_t)D^hX_tdt+\sigma d\left(\int^{S_t}_{0}h_s ds\right),\\
D^hX_0=0.\end{array}\right. \label{3.1}
\end{equation}
\end{theorem}
\begin{proof} 
%Let $\psi^{h}_t$ be the solution of equation (\ref{3.1})
%and it is easy to verify that 
Let $\psi^{h}_t$ be the solution of equation (\ref{3.1}).
It is easy to verify that 
$\mathbb{E}\left[\sup_{s\leq t}|\psi^{h}_s|^{p}\right]\leq C$, where $C$ is a constant depending on $T,x, h$ and $p$. Then, we have
\begin{eqnarray*}
&&\frac{X^{\varepsilon  h}_{t}-X_t}{\varepsilon }-\psi^{h}_t\\
=\!\!\!\!\!\!\!\!&&\frac{1}{\varepsilon }\!\!\int^t_0\!\![b(X^{\varepsilon  h}_s, \alpha_s)\!\!
-b(X_s, \alpha_s)-\varepsilon \nabla b(X_s, \alpha_s)\psi^{h}_s]ds\\
=\!\!\!\!\!\!\!\!&&\int^t_0\left[\left(\int^1_0\nabla b(X_s+\nu(X^{\varepsilon  h}_s-X_s), \alpha_s)d\nu\right)
\frac{X^{\varepsilon  h}_s-X_s}{\varepsilon }-\nabla b(X_s, \alpha_s)\psi^{h}_s\right]ds\\
=\!\!\!\!\!\!\!\!&&\int^t_0\left(\int^1_0\nabla b(X_s+\nu(X^{\varepsilon  h}_s-X_s), \alpha_s)d\nu\right)
\left(\frac{X^{\varepsilon  h}_s-X_s}{\varepsilon }-\psi^{h}_s\right)ds+\varphi^{\varepsilon h}_t,
\end{eqnarray*}
where $\varphi^{\varepsilon h}_t$ is defined by
$$
\varphi^{\varepsilon h}_t=\int^t_0\left(\int^1_0\nabla b(X_s+\nu(X^{\varepsilon  h}_s-X_s),
\alpha_s)d\nu-\nabla b(X_s, \alpha_s)\right)\psi^{h}_s ds.
$$
By the condition \textbf{(H2)}, we obtain
\begin{eqnarray*}
\mathbb{E}\left[\sup_{s\leq t}\left|\frac{X^{\varepsilon  h}_{s}-X_s}{\varepsilon }-\psi^{ h}_s\right|^p\right]
\leq \!\!\!\!\!\!\!\!&&C\left(\mathbb{E}\sup_{s\leq t}|X^{\varepsilon  h}_s-X_s|^{2p}\right)^{1/2}
\left(\mathbb{E}\sup_{s\leq t}|\psi^{h}_s|^{2p}\right)^{1/2}\\
&&+C\int^t_0\mathbb{E}\left|\frac{X^{\varepsilon  h}_s-X_s}{\varepsilon }-\psi^{h}_s\right|^p ds.
\end{eqnarray*}
By Lemmas \ref{l.3.5} and Gronwall's inequality, we obtain
$$
\lim_{\varepsilon \rightarrow 0}\mathbb{E}\left[\sup_{s\leq t}\left|\frac{X^{\varepsilon  h}_{s}-X_s}{\varepsilon }
-\psi^{h}_s\right|^p\right]=0.
$$
This implies that for $p\geq 2$,
$$
\lim_{\varepsilon \rightarrow 0}\mathbb{E}_{1}\left\|\frac{X^{\varepsilon  h}_{t}-X_t}
{\varepsilon }-\psi^{h}_t\right\|_{\mathbb{R}^n\otimes V}^p=0.
$$
Now, let $D_{s}X_t$ be the solution of the following equation:
\begin{eqnarray*}
D_{s}X_t=\!\!\!\!\!\!\!\!&&\sigma+\int^t_{0}\nabla b(X_r, \alpha_r)D_{s}X_r dr,\quad s\leq S_t,
\end{eqnarray*}
and $D_{s}X_t=0$ for $s>S_t$. Then we can easily obtain that $D^{h}X_t=\psi^{h}_t$ and
$DX_t\in L^{\infty-}(\Omega_{1}, H\otimes\mathbb{R}^n\otimes V)$. Hence, $X_t\in \mathbb{D}^{1,\infty}(\mathbb{R}^n\otimes V)$.
The proof is complete.
\end{proof}

\begin{remark} \label{r.4.5}Following the same idea as the above we can prove that if the function $b(x,i)$
is infinitely differentiable in $x$ with bounded partial derivatives of all orders, then $ X_t\in\mathbb{D}^{\infty}(\mathbb{R}^n\otimes V)$.
\end{remark}

%Similar as we did above, it is easy to show the following chain rule.
Using argument similar as above, one can easily prove the following chain rule.

\begin{theorem} \textbf{(Chain rule)} Assume that  condition \textbf{(H2)} holds.
Then for any $h\in H$, $t\ge 0$ and  $p\geq 2$, if  $f\in C^{2}_b(\mathbb{R}^n\times\mathbb{S})$, we have
$$
\lim_{\varepsilon \rightarrow 0}\mathbb{E}\left|\frac{f(X^{\varepsilon  h}_{t}, \alpha_t)
-f(X_t, \alpha_t)}{\varepsilon }-\nabla f(X_t, \alpha_t)D^{h}X_t\right|^p=0.
$$
Moreover, $f(X_t, \alpha_t)\in \mathbb{D}^{1, \infty}(V)$ and $Df(X_t, \alpha_t)=\nabla f(X_t, \alpha_t)DX_t$.
\end{theorem}

\subsection{Malliavin covariance matrix }

\begin{definition}
Suppose that $F(x, \alpha):\Omega\rightarrow \RR^n$ is a random vector
for all $x\in \RR^n$ and $\al\in \mathbb{S}$.  We say that its  gradient with respect to $x$ exists  (in the mean square sense)
if there is $A(x,\al): \Omega\rightarrow \RR^{n^2}$ such that for any $\xi\in \RR^n$ such that
\[
\lim_{\varepsilon \rightarrow 0}\mathbb{E}\left|\frac{F(x+\varepsilon \xi, \alpha)
-F(x, \alpha)}{\varepsilon }-A(x,\al)\xi \right|^{2}=0\,.
\]
We denote the gradient matrix $A(x,\al)$ by $\nabla  F(x, \alpha)$.
\end{definition}

By an argument similar to that used in the proof of Theorem \ref{main},
we can obtain the following theorem.

\begin{theorem}
Assume condition \textbf{(H2)} holds. Let $\left\{X_{t}(x,\alpha),
 t\ge 0\right\}$ be the solution of equation (\ref{e.x}), with $X_0=x, \alpha_0=\alpha$. Then  the gradient  of $X_{t}(x,\alpha)$ with respect to $x$  (in the mean square sense) exists.  If  we denote
\[
J_{t}  := \nabla X_{t}(x,\alpha)\,,
\]
then
\begin{equation}
J_t=I+\int^t_0\nabla b(X_s, \alpha_s)J_{s}  ds\,,\label{e.4.1}
\end{equation}
where $I$ is the $n$ dimensional identity matrix. Moreover, $J_{t}$ is invertible and its inverse $K_t$
satisfies
\begin{equation}
K_{t} = I-\int^t_0 K_{s}\nabla b(X_s, \alpha_s)ds.
\end{equation}
\end{theorem}

\vspace{0.3cm}
By Gronwall's inequality, we can easily obtain $\max\{\|J_t\|, \|K_t\|\}\leq e^{\|\nabla b\|_{\infty}t}$, where $\|A\|:=\sup_{\{x\in \RR^n: |x|=1\}}|\langle A, x\rangle|$, for any $A\in \RR^{n^2}$, and $\|\nabla b\|_{\infty}=\sup_{(x,\alpha)\in \RR^n\times \mathbb{S}}\|\nabla b(x,\alpha)\|$. Using the integration by parts formula, we have
$$
K_t D^{h}X_t=\int^{t}_0 K_s \sigma d\left(\int^{S_s}_0h_r dr\right).
$$

The Malliavin covariance matrix $M_t$ is defined by:
$$
M_t:=\langle DX_t \,, (DX_t)^*\rangle _H.
$$
Using the method developed in \cite[Theorem 3.3]{KS},one can easily show that
$$
M_t=J_t\int^{t}_0 K_s \sigma \sigma^{\ast}K^{\ast}_s dS_s J^{\ast}_t,
$$
where $J^{\ast}_{t}$, $\sigma^{\ast}$ and $K^{\ast}_s$ are the matrix transposes  of
$J_{t}$, $\sigma$ and $K_s$ respectively.

\section{Smooth density}

In this section, we will prove that the random vector $X_t$ has a smooth density under suitable assumptions on the coefficients. To this end, 
%we first study the 
we first prove a 
Norris-type lemma on time interval. Then we use it to show that, under a uniform H\"{o}rmander's condition, the determinant of the Malliavin covariance matrix of $X_t$ has finite negative moments of all orders. Finally, we prove that $X_t$ has a smooth density by considering the small jumps and the large jumps separately.

\subsection{Norris type lemma on time interval}
The following lemma is a called a Norris type lemma and 
plays a key role in proving that the Malliavin covariance matrix $M_t$ has finite negative moments of all orders. The classical Norris lemma (e.g., see \cite[Lemma 2.3.2]{N}) is for the continous case. 
%However, the solution $X_t$ is a jump process now, for this reason, 
Since we are we are dealing SDEs driven by an discontinuous subordinate Brownian motion, 
we will use a form of the Norris' type lemma for jump processes developed in \cite{Z1}. For our purpose, we will prove that this kind of Norris type lemma also holds on time interval.

In spirit of \cite{Z1}, we need the following condition:

\medskip
\noindent \textbf{(H3)} There exist constants $\theta\in (0,2)$ and $c_{\theta}>0$, such that
\begin{equation}
\lim_{\varepsilon\to 0}\varepsilon^{\frac{\theta}{2}-1}\int^{\varepsilon}_{0} u\nu_{S}(du)=c_{\theta}>0.\label{5.2}
\end{equation}

\begin{remark}
Let $\nu_{S}(du)=u^{-(1+\al/2)}du$ be the L\'evy measure of $\al/2$-stable subordinator. It is easy to see that (\ref{5.2}) holds for $\theta=\al$.
\end{remark}

\begin{lemma}\label{N.R.} (\textbf{Norris type lemma})
Assume that condition \textbf{(H3)} holds. Let $V(x, i): \RR^n\times \mathbb{S}\rightarrow \RR^n\times\RR^n$ be infinitely
differentiable in $x$ with bounded partial derivative of all orders. For any $i\in\mathbb{S}$, $p\geq 2$, $\beta\in(0\vee(4\theta-7), 1)$,
$0\leq t_1<t_2\leq 1$, there exists $\varepsilon_0=(t_2-t_1)^{C(\beta, p)}\varepsilon(p)$, where $C(\beta, p)$ and $\varepsilon(p)$
are two positive constants dependent on $\beta, p$ and $p$ respectively, such that for all $\varepsilon\in (0, \varepsilon_0)$,
\begin{eqnarray*}
\sup_{|v|=1}\PP\left( \int^{t_2}_{t_1}|v^{\ast}K^{(i)}_s[b, V](X^{(i)}_s, i)|^2 ds\geq \varepsilon^{\frac{1-\beta}{18-\beta}},
\int^{t_2}_{t_1}|v^{\ast}K^{(i)}_s V(X^{(i)}_s, i)|^2 ds\leq \varepsilon\right)\leq \varepsilon^p, \label{N.L.}
\end{eqnarray*}
where $[b,V](x,i)=b(x,i)\cdot\nabla V(x,i)-V(x,i)\cdot \nabla b(x,i)$, $X^{(i)}_t$ and $K^{(i)}_t$ satisfy the following two equations respectively:
\begin{eqnarray*}
X^{(i)}_t=X_{t_1}+\int^{t_2}_{t_1}b(X^{(i)}_r, i)dr+\sigma (L_{t}-L_{t_1}),\quad t_1\leq t\leq t_2
\end{eqnarray*}
and
\begin{eqnarray*}
K^{(i)}_t=K_{t_1}-\int^{t_2}_{t_1}K^{(i)}_r \nabla b(X^{(i)}_r, i)dr,\quad t_1\leq t\leq t_2.
\end{eqnarray*}
\end{lemma}
\begin{proof}
We first show that for any fixed $i\in \mathbb{S}$, $\beta\in(0\vee(4\theta-7), 1)$, $0\leq t_1<t_2\leq 1$, there exist
two constants $C_1\geq 1$ and $C_2\in (0, 1)$ such that for all $\delta\in(0, 1)$,
\begin{eqnarray}
&&\sup_{|v|=1}\PP\left(\int^{t_2}_{t_1}|v^{\ast}K^{(i)}_s[b, V](X^{(i)}_s, i)|^2 ds\geq (t_2-t_1)\delta^{\frac{1-\beta}{2}},\right.\nonumber\\
&&\quad\quad\quad\quad\left.\int^{t_2}_{t_1}|v^{\ast}K^{(i)}_s V(X^{(i)}_s, i)|^2 ds\leq (t_2-t_1)\delta^{9-\frac{\beta}{2}}\right)\nonumber\\
&\leq& C_1\exp\{-C_2(t_2-t_1)\delta^{-\frac{\beta}{2}}\}.\label{e.5.13}
\end{eqnarray}
In fact, by a changing of variables, we have that, for any $v\in \mathbb{R}^n$,
\begin{eqnarray}
&&\sup_{|v|=1}\PP\left( \int^{t_2}_{t_1}|v^{\ast}K^{(i)}_s[b, V](X^{(i)}_s, i)|^2 ds\geq (t_2-t_1)\delta^{\frac{1-\beta}{2}},\right.\nonumber\\
&&\quad\quad\quad\left.\int^{t_2}_{t_1}|v^{\ast}K^{(i)}_s V(X^{(i)}_s, i)|^2 ds\leq (t_2-t_1)\delta^{9-\frac{\beta}{2}}\right)\nonumber\\
&=&\sup_{|v|=1}\PP\left( \int^{t_2-t_1}_{0}|v^{\ast}\tilde{K}^{(i)}_s[b, V](\tilde{X}^{(i)}_s, i)|^2 ds\geq
(t_2-t_1)\delta^{\frac{1-\beta}{2}}, \right.\nonumber\\
&&\quad\quad\quad\left.\int^{t_2-t_1}_{0}|v^{\ast}\tilde{K}^{(i)}_s V(\tilde{X}^{(i)}_s, i)|^2 ds
\leq (t_2-t_1)\delta^{9-\frac{\beta}{2}}\right),\label{4.4}
\end{eqnarray}
where $\tilde{K}^{(i)}_s:= K^{(i)}_{t_1+s}$, $\tilde{X}^{(i)}_s:=X^{(i)}_{t_1+s}$, for $0\leq s\leq t_2-t_1$. Obviously,
\begin{eqnarray*}
\tilde{X}^{(i)}_s=X_{t_1}+\int^{s}_{0}b(\tilde{X}^{(i)}_r, i)dr+\sigma \tilde{L}_s,\quad 0\leq s\leq t_2-t_1
\end{eqnarray*}
and
\begin{eqnarray*}
\tilde{K}^{(i)}_s=K_{t_1}-\int^{s}_{0}\tilde{K}^{(i)}_r \nabla b(\tilde{X}^{(i)}_r, i)dr,\quad 0\leq s\leq t_2-t_1,
\end{eqnarray*}
where $\tilde{L}_s:=L_{t_1+s}-L_{t_1}$ is also a L\'{e}vy process and has the same distribution of $L_s$.

%The estimate of (\ref{4.4}) is changed into an
The estimate  (\ref{4.4}) is now changed into an
estimate for general SDE driven by subordinate Brownian motions without switching. Hence, applying \cite[Lemma 5.1]{Z2}, it is easy to see that (\ref{e.5.13}) holds.

Now, we set $\varepsilon:=(t_2-t_1)\delta^{9-\frac{\beta}{2}}$. Then noticing that $\varepsilon^{\frac{1-\beta}{18-\beta}}
\geq (t_2-t_1)\delta^{\frac{1-\beta}{2}}$, and by (\ref{e.5.13}), for any $\varepsilon\in (0, t_2-t_1)$, we have
\begin{eqnarray*}
&&\sup_{|v|=1}\PP\left( \int^{t_2}_{t_1}|v^{\ast}K^{(i)}_s[b, V](X^{(i)}_s, i)|^2 ds\geq \varepsilon^{\frac{1-\beta}{18-\beta}},
\int^{t_2}_{t_1}|v^{\ast}K^{(i)}_s V(X^{(i)}_s, i)|^2 ds\leq \varepsilon\right)\\
\leq\!\!\!\!\!\!\!\!&&C_1\exp\{-C_2(t_2-t_1)^{\frac{18}{18-\beta}}\varepsilon^{-\frac{\beta}{18-\beta}}\}.
\end{eqnarray*}
Therefore, there exists $\varepsilon_0=(t_2-t_1)^{C(\beta, p)}\varepsilon(p)$, where $C(\beta, p)$ and $\varepsilon(p)$ are two
positive constants dependent on $\beta, p$ and $p$ respectively, such that for all $\varepsilon\in(0, \varepsilon_0)$,
(\ref{N.L.}) holds. The proof is complete.
\end{proof}

\subsection{Main result}

\vspace{0.3cm}
Now we are going to study the smoothness of the density for $X_t$. The difficulty in our current situation
is that $b$ depends on the switching process $\al_t$. Following the idea in \cite{HNSX}, for any fixed $t>0$, define $N_t:=N([0, t], m(m-1)K)$, so $N_t$ is a Poisson process with parameter $m(m-1)K$, conditioned on the number of jumps of the Poisson process up to time $t$, that is, $N_t=k$,
there exists a  random interval $[T_1, T_2)$ with $0\leq T_1< T_2\leq t$ such that $T_2-T_1\geq\frac{t}{k+1}$ and
$\alpha_t=\alpha_{T_1}$ for all $ t\in[T_1, T_2)$ (because that the jump times of $\al_t$ is a subsequence of the jump times of $N_t$).
On this time interval, we will use the Lemma \ref{N.R.} above. 

First, we make the following assumption:

\noindent \textbf{(H4)} (Uniform H\"{o}rmander type condition) There exists some $j_0\in\mathbb{N}_{+}$, such that
\begin{eqnarray}
\inf_{(x, i)\in\RR^n\times \mathbb{S}}\inf_{|v|=1}\sum^{j_0}_{j=1}|v^{\ast}B_j(x, i)|^2=:\kappa_1>0,\label{UHC}
\end{eqnarray}
where $B_1(x,i)=\sigma$ and $B_{j+1}(x,i):=[b, B_j](x, i)$ for $j\in\mathbb{N}_{+}$.

%For the reason of the technique, 
For technical reasons, 
we will divide the proof into two subsections, i.e., by considering the small jumps and the large jumps separately.

\subsubsection{\bf If $S_t$ has finite moments of all orders}
In this section, we suppose that $S_t$ has finite moments of all orders and
$b\in C^{\infty}(\RR^n\times \mathbb{S})$ has bounded derivatives of all orders.

\begin{lemma} \label{l.5.3}For any $m, k\in \mathbb{N}_{+}$ with $m+k\geq 1$ and $p\geq 1$, we have
\begin{eqnarray}
\sup_{(x, \alpha)\in \RR^n\times \mathbb{S}}\EE\left( \|D^m \nabla^{k}X_t(x,\al)\|^{p}_{\HH^{\otimes m}\otimes \RR^{n^k}}\right)<+\infty. \label{4.17}
\end{eqnarray}
\end{lemma}

\begin{proof}
By Theorem \ref{main} and $\|J_t\|\leq e^{\|\nabla b\|_{\infty}t}$, we know that (\ref{4.17}) holds for $m+k=1$. For general $m$ and $k$, it follows by
similar calculations and induction method.
\end{proof}

\begin{theorem} \label{Th 5.4} Assume that conditions \textbf{(H3)} and \textbf{(H4)} hold. Then the Malliavin matrix $M_t$ is invertible
$\mathbb{P}$-a.s. and $\det(M^{-1}_t)\in L^p(\Omega)$ for all $p\geq 2$, $t\in(0, 1]$.
\end{theorem}

\begin{proof}
We recall that $M_t=J_{t}Q_t J_{t}^{\ast}$, where $Q_t:=\int^{t}_0 K_s \sigma \sigma^{\ast} K^{\ast}_s d S_s$.
It suffices to prove $\det(Q^{-1}_t)\in L^p(\Omega)$ for all $p\geq 2$.

Recall that $\{N_t=N([0, t], m(m-1)K)\}$  is a Poisson process with parameter $\lambda:=m(m-1)K$.
For a fixed $0<t\leq 1$, conditioned on  $N_t=k$, there exists a  random interval $[T_1, T_2]\subset [0, 1]$  such that $T_2-T_1\geq\frac{t}{k+1}$
and $\alpha_s=\alpha_{T_1}$ for all $ s\in[T_1, T_2)$.

By \cite[lemma 3.1]{Z1} and for the given $\theta$ in condition \textbf{(H3)}, and using the fact that the Poisson process
$N_t$ is independent of $W, S$, for any $p\geq 2$, there exists an $\varepsilon_0=\varepsilon_0(\theta, p)>0$ such that
for all $\varepsilon\in(0, \varepsilon_0)$,
\begin{eqnarray}
&&\mathbb{P}\left\{\int^{T_2}_{T_1}|v^{\ast}K_s\sigma|^2 d S_s\leq \varepsilon\big| N_t=k\right\}\nonumber\\
\leq\!\!\!\!\!\!\!\!&&\mathbb{P}\left\{\int^{T_2}_{T_1}|v^{\ast}K_s\sigma|^2 d S_s\leq \varepsilon, \int^{T_2}_{T_1}
|v^{\ast}K_s\sigma|^2 ds\geq \varepsilon^{\theta/4}\big| N_t=k\right\}\nonumber\\
&&+\mathbb{P}\left\{\int^{T_2}_{T_1}|v^{\ast}K_s\sigma|^2 ds< \varepsilon^{\theta/4}\big| N_t=k\right\}\nonumber\\
\leq\!\!\!\!\!\!\!\!&&\exp\left\{1-\frac{1}{2\varepsilon^{1-\theta/4}}\int^{C\varepsilon}_{0}u\nu_S(du)\right\}
+\mathbb{P}\left\{\int^{T_2}_{T_1}|v^{\ast}K_s\sigma|^2 ds< \varepsilon^{\theta/4}\big| N_t=k\right\}\nonumber\\
\leq\!\!\!\!\!\!\!\!&&\exp\left\{1-\varepsilon^{-\theta/8}\right\}+\mathbb{P}\left\{\int^{T_2}_{T_1}|v^{\ast}K_s\sigma|^2 ds<
\varepsilon^{\theta/4}\big| N_t=k\right\}\nonumber\\
\leq\!\!\!\!\!\!\!\!&&\varepsilon^p+\mathbb{P}\left\{\int^{T_2}_{T_1}|v^{\ast}K_s\sigma|^2 ds<
\varepsilon^{\theta/4}\big| N_t=k\right\}.\label{5.5}
\end{eqnarray}

Now, for any fixed $\beta\in(0\vee(4\theta-7), 1)$, $j=1, 2,\ldots, j_0$, denote $m(j)=(\frac{18-\beta}{1-\beta})^{j-1}$ and define
$$
E_j:=\left\{\int^{T_{2}}_{T_1}|v^{\ast}K_s B_j(X_s, \al_s)|^2 ds<\varepsilon ^{\frac{m_j\theta}{4}}\right\}.
$$
Clearly, $\{\int^{T_2}_{T_1}|v^{\ast}K_s\sigma|^2 ds< \varepsilon^{\theta/4}\}= E_1$. Consider the decomposition
$$E_1\subseteq (E_1 \cap E_2 ^{c})\cup(E_2\cap E_3^{c})\cup\cdots\cup(E_{j_0-1}\cap E_{j_0}^{c})\cup F,$$
where $F=E_1\cap E_2\cap\cdots\cap E_{j_0}$. Then for any unit vector $v$ we have
\begin{eqnarray}
\mathbb{P}\left\{\int^{T_2}_{T_1}|v^{\ast}K_s\sigma|^2 ds<\varepsilon^{\theta/4}\big| N_t=k \right\}=\!\!\!\!\!\!\!\!&&\mathbb{P}(E_1| N_t=k)\nonumber\\
\leq\!\!\!\!\!\!\!\!&& \mathbb{P}(F| N_t=k)+\sum^{j_0-1}_{j=1}\mathbb{P}(E_j\cap E_{j+1}^{c}| N_t=k)\,.
\label{e.4.7}
\end{eqnarray}
We are going to estimate each term in the above sum. This will be done in two steps.

\textit{\bf Step 1}: First we claim that when  $\varepsilon $ is sufficiently small,
the intersection of $F$ and $\{N_t=k\}$ is empty. In fact, taking into account (\ref{UHC}), on $N_t=k$, we have
\begin{eqnarray*}
F\subset\!\!\!\!\!\!\!\!&& \left\{\sum^{j_0}_{j=1}\int^{T_{2}}_{T_1}|v^{\ast}K_s B_j(X_s, \al_s)|^2ds
\leq j_0\varepsilon ^{\frac{m_{j_0}\theta}{4}}\right\}\\
=\!\!\!\!\!\!\!\!&&\left\{\sum^{j_0}_{j=1}\int^{T_2}_{T_1}
\left(\frac{|v^{\ast}K_s B_j(X_s, \al_s)|}{|v^{\ast}K_s|}\right)^2
|v^{\ast}K_s|^2 ds\leq j_0\varepsilon ^{\frac{m_{j_0}\theta}{4}}\right\}\\
\subset\!\!\!\!\!\!\!\!&&\left\{\frac{\kappa_1 t}{(k+1)e^{2\|\nabla b\|_{\infty}}}\leq j_0\varepsilon ^{\frac{m_{j_0}\theta}{4}}\right\}, \label{e.4.9}
\end{eqnarray*}
because that $|v^{\ast}K_s|\geq \frac{1}{\|J_{s}\|}\geq \frac{1}{e^{\|\nabla b\|_{\infty}}}$, for any $s\in(0,1]$. Thus $F\cap\{N_t=k\}=\emptyset$,
provided $\varepsilon < \varepsilon_1:=\left(\frac {\kappa_1 t}{j_0 e^{2\|\nabla b\|_{\infty}}(k+1)}\right)^{ \frac 4{ m(j_0)\theta}}$.

\medskip
\textit{\bf Step 2}: We shall bound the second terms in (\ref{e.4.7}). For any $j=1, 2, \ldots, j_0-1$, we have
\begin{eqnarray*}
\mathbb{P}(E_j\cap E_{j+1}^{c}| N_t=k)=\!\!\!\!\!\!\!\!&&
\mathbb{P}\left\{\int^{T_{2}}_{T_1}|v^{\ast}K_s B_j(X_s, \al_s)|^2 ds<\varepsilon ^{\frac{m_{j}\theta}{4}}\right.,\\
 &&\hspace{1cm}\left.\int^{T_{2}}_{T_1}|v^{\ast}K_s B_{j+1}(X_s, \al_s)|^2 ds\geq \varepsilon ^{\frac{m_{j+1}\theta}{4}}| N_t=k\right\}\\
=\!\!\!\!\!\!\!\!&& \mathbb{P}\left\{\int^{T_{2}}_{T_{1}}|v^{\ast}K^{(\al_{T_1})}_{s}B_j(X^{(\al_{T_1})}_s, \alpha_{T_1})|^2ds
\leq \Big(\varepsilon ^{\frac{m_{j+1}\theta}{4}}\right)^{\frac{1-\beta}{18-\beta}},\\
&&\left.\hspace{1cm}\int^{T_{2}}_{T_1}|v^{\ast}K^{(\al_{T_1})}_s B_{j+1}(X^{(\al_{T_1})}_s, \al_{T_1})|^2 ds
\geq \varepsilon ^{\frac{m_{j+1}\theta}{4}}| N_t=k\right\}\,.
\end{eqnarray*}
Recall that $T_2-T_1\geq\frac{t}{k+1}$ and that processes $N_t$ and  $L_t$ are independent, by using Lemma \ref{N.R.}, we obtain
\begin{equation}
\mathbb{P}(E_j\cap E_{j+1}^{c}| N_t=k)\leq\varepsilon^{p},\label{e.5.19}
\end{equation}
for $0<\varepsilon\leq \varepsilon_2=(\frac{t}{k+1})^{C(p)}\varepsilon(p)$, where $C(p)$ and $\varepsilon(p)$ are two positive constants dependent on $p$.

Hence, by (\ref{5.5})-(\ref{e.5.19}), we have
$$
\mathbb{P}\{v^{\ast}Q_t v\leq \varepsilon| N_t=k \}\leq\mathbb{P}\left\{\int^{T_2}_{T_1}|v^{\ast}K_s\sigma|^2 d S_s
\leq \varepsilon\big| N_t=k\right\}\leq \varepsilon^{p}
$$
for $\varepsilon<\min\{\varepsilon_0, \varepsilon_1, \varepsilon_2\}$.
Then, following the steps of \cite[Lemma 2.3.1]{N}, we can obtain that
$$
\mathbb{P}\left\{\inf_{|v|=1}v^{\ast}Q_t v\leq \varepsilon\big| N_t=k\right\}\leq \varepsilon^{p}
$$
for all $0<\varepsilon\leq C_1(\frac{t}{k+1})^{C_2} $ and for all $p\geq 2$, where $C_1, C_2$ are two positive
constants depending on $p$ and $n$. By the fact that $\det(Q_t)\geq (\inf_{|v|=1}v^{\ast}Q_t v)^n$, we have
\begin{eqnarray}
\EE  |\det(Q_t)|^{-p}  &\leq & \EE  \left(\inf_{|v|=1}v^{\ast}Q_t v\right)^{-np}\nonumber \\
&\leq & \sum^{\infty}_{k=0}\mathbb{P}(N_t=k)\EE\left( (\inf_{|v|=1}v^{\ast}Q_t v)^{-np}\big|N_t=k\right)\nonumber\\
&\leq &\sum^{\infty}_{k=0}\frac{\lambda^k}{k!}e^{\lambda}\left[  C_1 \left( \frac t{k+1} \right)^{C_2}
+  \frac{1}{C_1} \left( \frac {k+1}{t} \right)^{C_2} \right] <\infty.\label{M_t}
\end{eqnarray}
The proof is now complete.
\end{proof}

Now we can prove the following gradient estimate.

\begin{theorem}
For any $k, m\in\mathbb{N}$ with $k+m\geq 1$, there are $\gamma_{k, m}>0$ and $C=C_{k, m}>0$ such that for all
$f\in C^{\infty}_{b}(\RR^n\times \mathbb{S})$ and $t\in(0,1)$,
\begin{eqnarray}
|\nabla^{k}\EE(\nabla^m f)(X_t, \alpha_t)|\leq C\|f\|_{\infty}t^{-\gamma_{k, m}}. \label{4.18}
\end{eqnarray}
\end{theorem}

\begin{proof}
By the chain rule, we have
\begin{eqnarray*}
\nabla^{k}\EE(\nabla^m f)(X_t, \alpha_t)=\!\!\!\!\!\!\!\!&&\sum^{k}_{j=1}\EE\left((\nabla^{m+j}f)(X_t,\alpha_t)
G_j(\nabla X_t,\ldots, \nabla^{k}X_t)\right)
\end{eqnarray*}
where $\{G_j, j=1,\ldots, k\}$ are real polynomial functions. By the duality relationship, the chain rule, Lemma \ref{l.5.3} and H\"{o}lder's
inequality, through cumbersome calculations (for details see the argument in \cite[Proposition 2.1.4]{N}), one finds that there exist integer
$p=p_{k, m}$, $C=C_{k,m}>0$ and $\gamma_{k, m}>0$ such that for all $t\in(0,1)$,
$$
|\nabla^{k}\EE(\nabla^m f)(X_t, \alpha_t)|\leq C\|f\|_{\infty}\EE|(\text{det} M_t)^{-1}|^{p}
\leq C\|f\|_{\infty}t^{-\gamma_{k, m}},
$$
where the last inequality follows by (\ref{M_t}). The proof is complete.
\end{proof}

\subsubsection{\bf Without the assumption of $S_t$ has finite moments of all orders.}

Let $S'_t$ be a subordinator with L\'{e}vy measure
$1_{(0,1)}\nu_{S}(du)$ and independent of $(W_t)_{t\geq 0}$ and $N(dt,dy)$. Let $X'_t$ solve the following equations:
\begin{equation}
dX'_t=b(X'_t, \alpha_t)dt+\sigma dW_{S'_t},~~~~\ (X'_0\,,\al_0)=(x,\al)\in \RR^n \times \mathbb{S} \,,
\end{equation}
where $\alpha_t$ is the one in $(\ref{alpha})$. Let us write
$$
P'_tf(x, \al):=\EE f(X'_t(x), \al_t(\al)).
$$
Notice that $S'_t$ has finite moments of all order, then the results above hold for the process $(X'_t,\al_t)$.
We now proceed to find the relation between the semigroups $P_t$ and $P'_t$, so that we can estimate the semigroup $P_t$
via $P'_t$.

Following the steps in \cite[Section 3.3]{Z1}, we first give two lemmas whose proofs are almost the same
with Lemmas 3.9 and 3.10 in \cite{Z1}, so we omit the proof.

\begin{lemma}
Let $f\in C^{\infty}_b(\RR^n\times \mathbb{S})$. For any $m\in \mathbb{N}$, there exists a constant $C_m\geq 1$
such that for all $(x, \al)\in\RR^n\times \mathbb{S}$,
\begin{equation}
|\nabla^{m}P'_tf(x,\al)|\leq C_m\sum^{m}_{k=1}P'_t|\nabla^{k}f|(x,\al).\label{4.22}
\end{equation}
\end{lemma}

\begin{lemma}\label{l.5.7}
Let $J'_t:=\nabla X'_t(x)$ and $K'_t$ be the inverse of matrix $J'_t$. Let $f\in C^{\infty}_b(\RR^n\times \mathbb{S})$.
Then for any $j=1,\ldots, n$, we have the following formula:
\begin{equation}
P'_t(\partial_{j}f)(x,\al)= \text{div}~~Q^{\cdot j}(t,x,\al; f)-G^{j}(t, x,\al; f),
\end{equation}
where $\text{div} f(x)=\sum^{n}_{i=1}\partial_{x_i} f_i(x)$, for any $f(x)=(f_1(x), f_2(x), \ldots, f_n(x))\in C^{\infty}_b(\RR^n, \RR^n)$, and
\begin{equation}
Q^{i j}(t,x,\al; f):=\EE\big(f(X'_t(x),\al'_t(\al))(K'_t)_{ij}\big)\label{Q}
\end{equation}
and
\begin{equation}
G^{j}(t, x,\al; f):=\EE\big(f(X'_t(x),\al'_t(\al))\text{div}(K'_t)_{\cdot j}\big).\label{G}
\end{equation}
\end{lemma}

Now, let $\{\tau'_1, \tau'_2,\ldots, \tau'_n,\ldots\}$ and $\{\xi_1,\xi_2,\ldots, \xi_n,\ldots\}$ be two independent
families of independent and identically distributed $\RR^{+}$-valued random variables and 
$\RR^d$-valued random variables respectively, which are also independent of
$(W_t, S'_t)_{t\geq 0}$ and $N(dt, du)$. We assume that $\tau'_1$ has the exponential distribution of parameter
$\lambda_1:=\nu_{S}([1,\infty))$ and $\xi_1$ has the distributional density
$$
\frac{1}{\nu_{S}([1,\infty))}\int^{\infty}_{1}(2\pi s)^{-d/2}e^{-\frac{|x|^{2}}{2s}}\nu_{S}(ds).
$$
Set $\tau'_0:=0$ and $\xi_0:=0$, define
$$
N'_t:=\max\{k\geq 0:\tau'_0+\tau'_1+\cdots+\tau'_k\leq t\}=\sum^{\infty}_{k=0}1_{\{\tau'_0+\tau'_1+\cdots+\tau'_k\leq t\}}
$$
and
$$
H_t:=\xi_0+\xi_1+\cdots+\xi_{N'_t}=\sum^{N'_t}_{j=0}\xi_j.
$$
Then $H_t$ is a compound Poisson process with L\'{e}vy measure
$$
\nu_{H}(\Gamma)=\int^{\infty}_1 (2\pi s)^{-d/2}\left(\int_{\Gamma}e^{-\frac{|y|^2}{2s}}dy\right)\nu_{S}(ds).
$$
Moreover, it is easy to see that $H_t$ is independent of $W_{S'_t}$ and
\begin{equation}
(\sigma W_{S_t})_{t\geq 0}\stackrel{(d)}=(\sigma W_{S'_t}+\sigma H_t)_{t\geq 0}.\label{4.24}
\end{equation}

Let $\hbar_t$ be a c\`{a}dl\`{a}g purely discontinuous $\RR^n$-valued function with finite many jumps and $\hbar_0=0$.
Let $(X^{\hbar}_t(x), \al_t(\al))$ solve the following equations:
\begin{equation}
dX^{\hbar}_t=b(X^{\hbar}_t, \alpha_t)dt+\sigma d W_{S'_t}+d\hbar_t,,\ (X^{\hbar}_0\,,
\al_0)=(x,\al)\in \RR^n \times \mathbb{S} \,.
\end{equation}
Let $k$ be the jump number of $\hbar$ before time $t$. Let $0=t_0<t_1<t_2<\cdots<t_k\leq t$ be the jump times of $\hbar$.
By the Markovian property of $(X^{\hbar}_t(x), \al_t(\al))$, we have the following formula:
$$
\EE f(X^{\hbar}_t(x), \al_t(\al))=P'_{t_1}\cdots\theta_{\Delta \hbar_{t_{k-1}}}P'_{t_k-t_{k-1}}
\theta_{\Delta \hbar_{t_k}}P'_{t-t_k}f(x,\al),
$$
where
$$
\theta_{y}f(x,\al):=f(x+y, \al).
$$
Now, by (\ref{4.24}) we have
$$
(X_t(x),\al_t(\al))\stackrel{(d)}=(X^{\hbar}_t(x),\al_{t}(\al))\big|_{\hbar=\sigma H}.
$$
Hence,
\begin{eqnarray*}
P_t f(x,\alpha)=\!\!\!\!\!\!\!\!&&\mathbb{E} f(X_t(x),\alpha_t(\alpha))\\
=\!\!\!\!\!\!\!\!&&\mathbb{E} \left(f(X^{\hbar}_t(x),\alpha_t(\alpha))\big|_{\hbar=\sigma H}\right)\\
=\!\!\!\!\!\!\!\!&&\sum^{\infty}_{k=0}\mathbb{E}\left(P'_{\tau_1}\cdots\theta_{\sigma \xi_{k-1}}P'_{\tau'_k}
\theta_{\sigma \xi_{k}}P'_{t-(\tau'_1+\cdots+\tau'_k)}f(x,\alpha), N'_t=k\right).
\end{eqnarray*}
In view of
$$
\{N'_t=k\}=\{\tau'_1+\cdots+\tau'_k\leq t<\tau'_1+\cdots+\tau'_k+\tau'_{k+1}\},
$$
we further have
\begin{eqnarray}
P_t f(x,\alpha)=\!\!\!\!\!\!\!\!&& \sum^{\infty}_{k=0}\bigg\{\int_{\sum_{i=1}^kt_i\leq t<\sum_{i=1}^{k+1}t_i}
\mathbb{E}(P'_{t_1}\cdots\theta_{\sigma \xi_{k-1}}P'_{t_k}\theta_{\sigma \xi_{k}}P'_{t-\sum_{i=1}^kt_i}f(x,\alpha))\nonumber\\
&&\quad\quad\times\lambda^{k+1}_1e^{-\lambda_1\sum_{i=1}^{k+1}t_i}d t_1\cdots d t_{k+1}\bigg\}+P'_t f(x,\alpha)\mathbb{P}(N'_t=0)\nonumber\\
=\!\!\!\!\!\!\!\!&&\sum^{\infty}_{k=1}\bigg\{\lambda^{k}_1e^{-\lambda_1 t}\int_{\sum_{i=1}^kt_i\leq t}
\mathbb{E} \emph{I}^{\sigma \xi}_{f}(t_1,\ldots, t_k, t, x, \alpha)d t_1\cdots d t_k\bigg\}\nonumber\\
&&+P'_t f(x,\alpha)e^{-\lambda_1 t},\label{4.27}
\end{eqnarray}
where $\xi:=(\xi_1,\ldots, \xi_k)$ and $\emph{I}^{~\textbf{y}}_{f}(t_1,\ldots, t_k, t, x, \alpha)=P'_{t_1}\cdots\theta_{y_{k-1}}P'_{t_k}\theta_{y_{k}}P'_{t-(t_1+\cdots+t_k)}f(x,\alpha)$,
with $\textbf{y}:=(y_1, \ldots, y_k)$.

\vspace{0.3cm}

Our main theorem is following:
\begin{theorem}\label{SM density}
Let $b\in C^{\infty}(\RR^n\times \mathbb{S}; \RR^n)$ with bounded partial derivatives of all orders. Suppose conditions \textbf{(H3)} and \textbf{(H4)} hold.
Then for any $t\in (0,1]$, $X_t$ has a smooth density with respect to the Lebesgue measure on $\RR^n$.
\end{theorem}

\begin{proof} In order to prove the smoothness of density for $X_t$. By \cite{N},
it suffices to show that for any $f\in C^{\infty}_b(\mathbb{R}^n)$, we have
$$|\mathbb{E}\nabla^m_{i_1,\ldots, i_m}f(X_t)|\leq C\|f\|_{\infty},~~~\forall m\geq 1, (i_1,\ldots, i_m)\in\{1,\ldots, n\}^m,$$
where $\nabla^m_{i_1,\ldots, i_m}=\frac{\partial^m}{\partial x_{i_1}\cdots \partial x_{i_m}}$ and $C$ depends on $t, x, (i_1,\ldots, i_m)$.
However, this can be easily obtained if we can establish the same gradient estimate as in (\ref{4.18}).

If we let $t_{k+1}:=t-(t_1+\cdots+t_k)$, then there exists at least one $j\in\{1,2,\ldots,k+1\}$ such that $t_j\geq\frac{t}{k+1}$.
Thus, by (\ref{4.22}) and (\ref{4.18}), we have
\begin{eqnarray*}
|\nabla \emph{I}^{~\textbf{y}}_{f}(t_1,\ldots, t_k, t, x, \al)|\leq \!\!\!\!\!\!\!\!&&C^{j-1}_{1}\|\nabla P'_{t_j}\cdots \theta_{y_{k-1}}P'_{t_k}
\theta_{y_k}P'_{t_{k+1}}f\|_{\infty}\\
\leq\!\!\!\!\!\!\!\!&& CC^{j-1}_{1}t^{-\gamma_{1,0}}_{j}\|P'_{t_{j+1}}\cdots \theta_{y_{k-1}}P'_{t_k}\theta_{y_k}P'_{t_{k+1}}f\|_{\infty}\\
\leq\!\!\!\!\!\!\!\!&& CC^{k}_{1}(\frac{t}{k+1})^{-\gamma_{1,0}}\|f\|_{\infty}.
\end{eqnarray*}
Hence, by (\ref{4.27}) we have
\begin{eqnarray}
|\nabla P_t f(x,\al)|\leq \!\!\!\!\!\!\!\!&&C\|f\|_{\infty}t^{-\gamma_{1,0}}
e^{- \lambda_1 t}\left[1+\sum^{\infty}_{k=1}\lambda^{k}_1 C^{k}_{1}(k+1)^{\gamma_{1,0}}\int_{\sum_{i=1}^kt_i\leq t}dt_1\cdots dt_k\right]\nonumber\\
=\!\!\!\!\!\!\!\!&&C\|f\|_{\infty}t^{-\gamma_{1,0}}e^{-\lambda_1 t}\left(\sum^{\infty}_{k=0}\lambda^{k}_1 C^k_1(k+1)^{\gamma_{1,0}}
\frac{t^k}{k!}\right)\nonumber\\
\leq\!\!\!\!\!\!\!\!&& C\|f\|_{\infty}t^{-\gamma_{1,0}}. \label{4.32}
\end{eqnarray}
Thus, we obtain (\ref{4.18}) with $k=1$ and $m=0$.

For $l, i=1,\ldots, n$, set $F^{0}_{l i}(x,\al):=1_{\{l=i\}}f(x,\al)$. Let us recursively define for $m=0,1,\ldots,k$,
$$
F^{(m+1)}_{li}(x,\al):=\sum^{n}_{j=1}Q^{ij}(t_{k+1-m}, x+y_k+\cdots+y_{k-m},\al; F^{(m)}_{lj})
$$
and
$$
R^{(m+1)}_{l}(x,\al):=\sum^{n}_{j=1}G^{j}(t_{k+1-m}, x+y_k+\cdots+y_{k-m},\al; F^{(m)}_{lj}),
$$
where $Q^{ij}$ and $G^{j}$ are defined by (\ref{Q}) and (\ref{G}). From these definitions, it easy to see that
$$
\|F^{(m+1)}_{li}\|_{\infty}\leq \sum^{n}_{j=1}\|F^{(m)}_{lj}\|_{\infty}\EE((K'_t)_{ij})
\leq C\sum^{n}_{j=1}\|F^{(m)}_{lj}\|_{\infty}\leq C n^{m}\|f\|_{\infty}
$$
and
$$
\|R^{(m+1)}_{l}\|_{\infty}\leq \sum^{n}_{j=1}\|F^{(m)}_{lj}\|_{\infty}\EE((K'_t)_{ij})\EE(\text{div}(K'_t)_{\cdot j})\leq C n^{m+1}\|f\|_{\infty}.
$$
By repeatedly using Lemma \ref{l.5.7}, we obtain
\begin{eqnarray*}
&&|\emph{I}^{~\textbf{y}}_{\partial_{l}f}(t_1,\ldots, t_k, t, x, \al)|\\
= \!\!\!\!\!\!\!\!&&\left|P'_{t_1}\cdots \theta_{y_{j-1}}P'_{t_j}\text{div} F^{(k+1-j)}_{l \cdot}(x,\al)-\sum^{k+1-j}_{m=1}
P'_{t_1}\cdots\theta_{y_{k-m}}P'_{t_{k+1-m}}R^{(m)}_{l}(x, \al)\right|\\
\leq \!\!\!\!\!\!\!\!&& C t^{-\gamma_{1,0}}_j \sum^{n}_{i=1}\|F^{(k+1-j)}_{li}\|_{\infty}+\sum^{k+1-j}_{m=1}\|R^{(m)}_{l}\|_{\infty}\\
\leq \!\!\!\!\!\!\!\!&& C (\frac{t}{k+1})^{-\gamma_{1,0}}\|f\|_{\infty}+C\|f\|_{\infty}.
\end{eqnarray*}
As estimating in (\ref{4.32}), we can obtain (\ref{4.18}) with $k=0$ and $m=1$. For the general $m$ and $k$,
the gradient estimate (\ref{4.18}) follows by similar calculations and induction method. The proof is complete.
\end{proof}

\medskip

\textbf{Acknowledgment}.
Xiaobin Sun is supported by the National Natural Science Foundation of China (11601196),
Natural Science Foundation of the Higher Education Institutions of Jiangsu Province (16KJB110006), Scientific Research Staring Foundation of Jiangsu Normal University (15XLR010). Yingchao Xie is supported by the National Natural Science Foundation of China (11771187). The Project is also funded by the Priority Academic Program Development of Jiangsu Higher Education Institutions.

\bibliographystyle{amsplain}

\end{document}